\documentclass[12pt]{article}%
\usepackage{graphicx}
\usepackage[intlimits]{amsmath}
\usepackage{latexsym}
\usepackage{amsfonts}
\usepackage{amssymb}%
\makeatletter
\newfont{\footsc}{cmcsc10 at 8truept}
\newfont{\footbf}{cmbx10 at 8truept}
\newfont{\footrm}{cmr10 at 10truept}
\newtheorem{theorem}{Theorem}

\newtheorem{lemma}[theorem]{Lemma}

\newenvironment{proof}[1][Proof]{\noindent{\textbf {#1}  }}  {\hfill$\Box$\bigskip}

\begin{document}

\title{The spectral radius of subgraphs of regular graphs}
\author{Vladimir Nikiforov\\Department of Mathematical Sciences, University of Memphis, \\Memphis TN 38152, USA}
\maketitle

\begin{abstract}
Let $\mu\left(  G\right)  $ and $\mu_{\min}\left(  G\right)  $ be the largest
and smallest eigenvalues of the adjacency matrix of a graph $G$. Our main
results are:

(i) Let $G$ be a regular graph of order $n$ and diameter $D.$ If $H$ is a
proper subgraph of $G,$ then%
\[
\mu\left(  G\right)  -\mu\left(  H\right)  >\frac{1}{nD}.
\]

(ii) If $G$ is a connected regular nonbipartite graph of order $n$ and
diameter $D$, then
\[
\mu\left(  G\right)  +\mu_{\min}\left(  G\right)  >\frac{1}{nD}.
\]

\textbf{Keywords:} \textit{smallest eigenvalue, largest eigenvalue, diameter,
connected graph, nonbipartite graph}

\end{abstract}

Our notation follows \cite{Bol98}, \cite{CDS80}, and \cite{HoJo88}. In
particular all graphs are defined on the vertex set $\left[  n\right]
=\left\{  1,...,n\right\}  $ and $\mu\left(  G\right)  $ and $\mu_{\min
}\left(  G\right)  $ stand for the largest and smallest eigenvalues of the
adjacency matrix of a graph $G$.

The aim of this note is to improve some recent results on eigenvalues of
subgraphs of regular graphs. Our main result is the following theorem.

\begin{theorem}
\label{th1}Let $G$ be a regular graph of order $n$ and diameter $D.$ If $H$ is
a proper subgraph of $G,$ then%
\begin{equation}
\mu\left(  G\right)  -\mu\left(  H\right)  >\frac{1}{nD}. \label{mainin}%
\end{equation}

\end{theorem}

This theorem combines the hypothesis of Theorem 4 in \cite{Nik07} with the
conclusion of Corollary 2.2 in \cite{Cio07}, thus improving both statements.
As a consequence we obtain the following result about regular nonbipartite graphs.

\begin{theorem}
\label{th2}If $G$ is a regular nonbipartite graph of order $n$ and diameter
$D$, then
\[
\mu\left(  G\right)  +\mu_{\min}\left(  G\right)  >\frac{1}{nD}.
\]

\end{theorem}

\bigskip

To prove Theorem \ref{th1} we borrow ideas from \cite{Cio07} and \cite{Nik07}.

Write $dist_{G}\left(  u,v\right)  $ for the length of a shortest path joining
the vertices $u$ and $v$ in a graph $G.$ We start with a seemingly known lemma
that we prove here for convenience.

\begin{lemma}
\label{le1}Let $G$ be a graph of order $n$ and diameter $D.$ Let $uv$ be an
edge of $G$ such that the graph $G-uv$ is connected. Then, for every $w\in
V\left(  G\right)  ,$
\begin{equation}
dist_{H}\left(  w,u\right)  +dist_{H}\left(  w,v\right)  \leq2D. \label{in3}%
\end{equation}

\end{lemma}

\begin{proof}
Let $w\in V\left(  H\right)  $ and let $P\left(  u,w\right)  $ and $P\left(
v,w\right)  $ be shortest paths joining $u$ and $v$ to $w$ in $H.$ Let
$Q\left(  u,x\right)  $ and $Q\left(  v,x\right)  $ be the longest subpaths of
$P\left(  u,w\right)  $ and $P\left(  v,w\right)  $ having no internal
vertices in common. If $s\in Q\left(  u,x\right)  $ or $s\in Q\left(
v,x\right)  ,$ we obviously have%
\begin{equation}
dist_{H}\left(  w,s\right)  =dist_{H}\left(  w,x\right)  +dist_{H}\left(
s,x\right)  . \label{eq1}%
\end{equation}
The paths $Q\left(  u,x\right)  ,$ $Q\left(  v,x\right)  $ and the edge $uv$
form a cycle in $G;$ write $k$ for its length. Assume that $dist\left(
v,x\right)  \geq dist\left(  u,x\right)  $ and select $y\in Q\left(
v,x\right)  $ with $dist_{H}\left(  x,y\right)  =\left\lfloor k/2\right\rfloor
.$ Let $R\left(  w,y\right)  $ be a shortest path in $G$ joining $w$ to $y;$
clearly the length of $R\left(  w,y\right)  $ is at most $D.$ If $R\left(
w,y\right)  $ does not contain the edge $uv,$ it is a path in $H$ and, using
(\ref{eq1}), we find that%
\begin{align*}
D  &  \geq dist_{G}\left(  w,y\right)  =dist_{H}\left(  w,y\right)
=dist_{H}\left(  w,x\right)  +\left\lfloor k/2\right\rfloor \\
&  =dist_{H}\left(  w,x\right)  +\left\lfloor \frac{dist_{H}\left(
x,u\right)  +dist_{H}\left(  x,v\right)  +1}{2}\right\rfloor \\
&  \geq dist_{H}\left(  w,x\right)  +\frac{dist_{H}\left(  x,u\right)
+dist_{H}\left(  x,v\right)  }{2}=\frac{dist_{H}\left(  w,u\right)
+dist_{H}\left(  w,v\right)  }{2},
\end{align*}
implying (\ref{in3}). Let now $R\left(  w,y\right)  $ contain the edge $uv.$
Assume first that $v$ occurs before $u$ when traversing $R\left(  w,y\right)
$ from $w$ to $y.$ Then
\begin{align*}
dist_{H}\left(  w,u\right)  +dist_{H}\left(  w,v\right)   &  \leq
2dist_{H}\left(  w,x\right)  +dist_{H}\left(  x,u\right)  +dist_{H}\left(
x,v\right) \\
&  \leq2\left(  dist_{H}\left(  w,x\right)  +dist_{H}\left(  x,v\right)
\right)  <dist_{G}\left(  w,y\right)  \leq2D,
\end{align*}
implying (\ref{in3}). Finally, if $u$ occurs before $v$ when traversing
$R\left(  w,y\right)  $ from $w$ to $y,$ then
\begin{align*}
D  &  \geq dist_{G}\left(  w,y\right)  \geq dist_{H}\left(  w,u\right)
+1+dist_{H}\left(  v,y\right) \\
&  =dist_{H}\left(  w,x\right)  +dist_{H}\left(  x,u\right)  +1+dist_{H}%
\left(  v,y\right)  =dist_{H}\left(  w,x\right)  +\left\lceil k/2\right\rceil
\\
&  \geq dist_{H}\left(  w,x\right)  +\frac{dist_{H}\left(  x,u\right)
+dist_{H}\left(  x,v\right)  }{2}=\frac{dist_{H}\left(  w,u\right)
+dist_{H}\left(  w,v\right)  }{2},
\end{align*}
implying (\ref{in3}) and completing the proof.
\end{proof}

\bigskip

\begin{proof}
[\textbf{Proof of Theorem \ref{th1}}]Since $\mu\left(  H\right)  \leq
\mu\left(  H^{\prime}\right)  $ if $H\subset H^{\prime}$, we may assume that
$H$ is a maximal proper subgraph of $G$, that is to say, $V\left(  H\right)
=V\left(  G\right)  $ and $H$ differs from $G$ in a single edge $uv$. Write
$d$ for the degree of $G$ and set $\mu=\mu\left(  H\right)  $.

The rest of the proof is split into two cases: (a)\emph{ }$H$ connected; (b)
$H$ disconnected.\bigskip

\textbf{Case (a): }$H$\emph{ }\textbf{is connected.}

Let $\mathbf{x}=\left(  x_{1},...,x_{n}\right)  $ be a unit eigenvector to
$\mu$ and let $x_{w}$ be a maximal entry of $\mathbf{x};$ we thus have
$x_{w}^{2}\geq1/n.$ We can assume that $w\neq v$ and $w\neq u.$ Indeed, if
$w=v,$ we see that
\[
\mu x_{v}=%
{\textstyle\sum\limits_{vi\in E\left(  G\right)  }}
x_{i}\leq\left(  d-1\right)  x_{v},
\]
and so $d-\mu\geq1,$ implying (\ref{mainin}).

We have%
\[
d-\mu=d%
{\textstyle\sum\limits_{i\in V\left(  G\right)  }}
x_{i}^{2}-2%
{\textstyle\sum\limits_{ij\in V\left(  G\right)  }}
x_{i}x_{j}=%
{\textstyle\sum\limits_{ij\in V\left(  G\right)  }}
\left(  x_{i}-x_{j}\right)  ^{2}+x_{u}^{2}+x_{v}^{2}.
\]

Assume first that $dist_{H}\left(  w,u\right)  \leq D-1.$ Select a shortest
path $u=u_{1},u_{2},\ldots,u_{k}=w$ joining $u$ to $w$ in $H.$ We see that
\begin{align*}
d-\mu &  =%
{\textstyle\sum_{ij\in V\left(  G\right)  }}
\left(  x_{i}-x_{j}\right)  ^{2}+x_{u}^{2}+x_{v}^{2}>%
{\textstyle\sum_{i=1}^{k-1}}
\left(  x_{u_{i}}-x_{u_{i+1}}\right)  ^{2}+x_{u}^{2}\\
&  \geq\frac{1}{k-1}\left(  x_{u_{i}}-x_{u_{i+1}}\right)  ^{2}+x_{u}^{2}%
=\frac{1}{k-1}\left(  x_{w}-x_{u}\right)  ^{2}+x_{u}^{2}\geq\frac{1}{k}%
x_{w}^{2}\geq\frac{1}{nD},
\end{align*}
completing the proof.

Hence, in view of Lemma \ref{le1}, we shall assume that
\[
dist_{H}\left(  w,u\right)  =dist_{H}\left(  w,v\right)  =D.
\]

Select shortest paths $P\left(  u,w\right)  $ and $P\left(  v,w\right)  $
joining $u$ and $v$ to $w$ in $H.$ Let $Q\left(  u,z\right)  $ and $Q\left(
v,z\right)  $ be the longest subpaths of $P\left(  u,w\right)  $ and $P\left(
v,w\right)  $ having no internal vertices in common. Clearly $Q\left(
u,z\right)  $ and $Q\left(  v,z\right)  $ have the same length. Write
$Q\left(  z,w\right)  $ for the subpath of $P\left(  u,w\right)  $ joining $z$
to $w$ and let
\[
Q\left(  u,z\right)  =u_{1},\ldots,u_{k},\text{ \ \ }Q\left(  v,z\right)
=v_{1},\ldots,v_{k},\text{ \ \ }Q\left(  z,w\right)  =w_{1},\ldots,w_{l},
\]
where $u_{1}=u,$ $u_{k}=v_{k}=w_{1}=z,$ $w_{l}=w,$ $k+l-2=D.$ As above, we see
that
\begin{align*}
d-\mu &  \geq%
{\textstyle\sum_{i=1}^{k-1}}
\left(  x_{v_{i}}-x_{v_{i+1}}\right)  ^{2}+x_{v}^{2}+%
{\textstyle\sum_{i=1}^{k-1}}
\left(  x_{u_{i}}-x_{u_{i+1}}\right)  ^{2}+x_{u}^{2}+%
{\textstyle\sum_{i=1}^{l-1}}
\left(  x_{w_{i}}-x_{w_{i+1}}\right)  ^{2}\\
&  \geq\frac{2}{D-l+2}x_{z}^{2}+\frac{1}{l-1}\left(  x_{w}-x_{z}\right)
^{2}\geq\frac{2}{D+l-1}x_{w}^{2}\geq\frac{1}{Dn},
\end{align*}
completing the proof.\bigskip

\textbf{Case (b): }$H$\emph{ }\textbf{is disconnected.}

Since $G$ is connected, $H$ is union of two connected graphs $H_{1}$ and
$H_{2}$ such that $u\in H_{1},$ $v\in H_{2}.$ Assume $\mu=\mu\left(
H_{1}\right)  ,$ set $\left\vert H_{1}\right\vert =k$ and let $\mathbf{x}%
=\left(  x_{1},...,x_{k}\right)  $ be a unit eigenvector to $\mu.$ Let $x_{w}$
be a maximal entry of $\mathbf{x};$ we thus have $x_{w}^{2}\geq1/k.$ Like in
the previous case, we see that $w\neq u.$ Select a shortest path
$u=u_{1},u_{2},\ldots,u_{l}=w$ joining $u$ to $w$ in $H_{1}.$ Since
$dist_{G}\left(  v,w\right)  \leq diam$ $G=D,$ we see that $l=dist_{H_{1}%
}\left(  u,w\right)  \leq D-1.$ As above, we have%
\begin{align*}
d-\mu &  =%
{\textstyle\sum_{ij\in V\left(  G\right)  }}
\left(  x_{i}-x_{j}\right)  ^{2}+x_{u}^{2}+x_{v}^{2}>%
{\textstyle\sum_{i=1}^{l-1}}
\left(  x_{u_{i}}-x_{u_{i+1}}\right)  ^{2}+x_{u}^{2}\\
&  \geq\frac{1}{l-1}\left(  x_{u_{1}}-x_{u_{k}}\right)  ^{2}+x_{u}^{2}%
=\frac{1}{l-1}\left(  x_{w}-x_{u}\right)  ^{2}+x_{u}^{2}\geq\frac{1}{l}%
x_{w}^{2}>\frac{1}{nD},
\end{align*}
completing the proof.
\end{proof}

\bigskip

\begin{proof}
[\textbf{Proof of Theorem \ref{th2}}]Let $\mathbf{x}=\left(  x_{1}%
,...,x_{n}\right)  $ be an eigenvector to $\mu_{\min}\left(  G\right)  $ and
let $U=\left\{  u:x_{u}<0\right\}  .$ Write $H$ for the maximal bipartite
subgraph of $G$ containing all edges with exactly one vertex in $U;$ note that
$H$ is a proper subgraph of $G$ and $\mu_{\min}\left(  H\right)  <\mu_{\min
}\left(  G\right)  .$ Hence,%
\[
\mu\left(  G\right)  +\mu_{\min}\left(  G\right)  >\mu\left(  G\right)
+\mu_{\min}\left(  H\right)  =\mu\left(  G\right)  -\mu\left(  H\right)  ,
\]
and the assertion follows from Theorem \ref{th1}.
\end{proof}

\bigskip

\textbf{Acknowledgment }A remark of Lingsheng Shi triggered the present note.

\end{document}